\documentclass[reqno]{amsart}

\usepackage{amssymb,amsmath,amsthm,xcolor,enumerate,hyperref}

\allowdisplaybreaks

\newtheorem{thrm}{Theorem}[section]

\newtheorem{lem}[thrm]{Lemma}
\newtheorem{prop}[thrm]{Proposition}

\theoremstyle{definition}

\newtheorem{exm}[thrm]{Example}

\newtheorem{rem}[thrm]{Remark}

\renewcommand{\iff}{\Leftrightarrow}

\newcommand{\impl}{\Rightarrow}
\newcommand{\id}{\mathrm{id}}

\newcommand{\cI}[1]{\mathcal I{(#1)}}
\newcommand{\cS}[1]{\mathcal S{(#1)}}

\newcommand{\cG}[1]{\mathcal G{(#1)}}
\newcommand{\Sg}[1]{\Sigma{(#1)}}
\newcommand{\cR}{\mathcal R}

\newcommand{\dom}[1]{\operatorname{\mathrm{dom}}{#1}}
\newcommand{\ran}[1]{\operatorname{\mathrm{ran}}{#1}}

\newcommand{\m}{{}^{-1}}

\newcommand{\dt}{\delta}
\newcommand{\io}{\iota}
\newcommand{\af}{\alpha}

\newcommand{\tl}{\tilde}
\newcommand{\sst}{\subseteq}

\newcommand{\mt}{\wedge}
\newcommand{\jn}{\vee}
\newcommand{\Jn}{\bigvee}

\newcommand{\rt}{\rtimes}

\begin{document}
	
	\title{Partial actions and an embedding theorem for inverse semigroups}
	
	\author{Mykola Khrypchenko}
	\address{Departamento de Matem\'atica, Universidade Federal de Santa Catarina, Campus Reitor Jo\~ao David Ferreira Lima, Florian\'opolis, SC,  CEP: 88040--900, Brazil}
	\email{nskhripchenko@gmail.com}
	
	
	\thanks{Supported by FAPESP of Brazil (process number: 2012/01554--7)}

	\begin{abstract}
     We give a simple construction involving partial actions which permits us to obtain an easy proof of a weakened version of L.~O'Carroll's theorem on idempotent pure extensions of inverse semigroups.
	\end{abstract}
	
	\maketitle
	
\section*{Introduction}

It was proved by D.~B.~McAlister in~\cite[Theorem 2.6]{McAlister74-II} that, up to an isomorphism, each $E$-unitary inverse semigroup is of the form $P(G,X,Y)$ for some McAlister triple $(G,X,Y)$. There are several alternative proofs of this fact (see, for example, the survey~\cite{Lawson-Margolis-P-th}), in particular, the one given in~\cite{KL} uses partial actions of groups on semilattices.

L. O'Carroll generalized in~\cite[Theorem 4]{O'Carroll77} (which is a reformulation of~\cite[Theorem 2.11]{O'Carroll75}) McAlister's $P$-theorem by showing that for any inverse semigroup $S$ and an idempotent pure congruence $\rho$ on $S$ there is a fully strict $L$-triple $(T,X,Y)$ with $T\cong S/\rho$ and $Y\cong E(S)$, such that $S\cong L_m(T,X,Y)$. This gives rise to a categorical equivalence as shown in \cite[Theorems~1.5, 4.1 and 4.4]{Lawson96}. Moreover, when $S$ is $E$-unitary and $\rho$ is the minimum group congruence on $S$, this coincides with McAlister's description of $S$ as $P(G,X,Y)$.

In this short note we prove the following weakened version of~\cite[Theorem~4]{O'Carroll77}: given an inverse semigroup $S$ and an idempotent pure congruence $\rho$ on $S$, there exists a fully strict partial action $\tau$ of $S/\rho$ on $E(S)$, such that $S$ embeds into $E(S)\rt_\tau(S/\rho)$. Moreover, the image of $S$ in $E(S)\rt_\tau(S/\rho)$ coincides with a subsemigroup $E(S)\rt_\tau^m(S/\rho)$ of $E(S)\rt_\tau(S/\rho)$ which can be explicitly described. Our approach permits one to avoid the necessity of the poset $X$ in O'Carroll's $L$-triple. On the other hand, this weakens the result, as the partial action $\tau$ that we construct may in fact be non-globalizable.

\section{Preliminaries}

Given an inverse semigroup $S$, we use the standard notation $\le$ for the natural partial order on $S$, and $\mt$, $\jn$ for the corresponding meet and join. Notice that $s\mt t$ or $s\jn t$ may not exist in $S$ in general. However, if $e$ and $f$ are idempotents of $S$, then $e\mt f$ exists and coincides with $ef$, which is again an idempotent. Thus, the subset of idempotents of $S$, denoted by $E(S)$, forms a (meet) semilattice with respect to $\mt$. 

A congruence $\rho$ on $S$ is said to be {\it idempotent pure}, if 
\begin{align*}
(e,s)\in\rho\ \&\ e\in E(S)\impl s\in E(S).
\end{align*}
The semigroup $S$ is called {\it $E$-unitary}, if its \emph{minimum group congruence} $\sigma$, defined by 
\begin{align*}
(s,t)\in\sigma\iff\exists u\le s,t,
\end{align*}
is idempotent pure. An {\it $F$-inverse monoid} is an inverse semigroup, in which every $\sigma$-class has a maximum element under $\le$. Each $F$-inverse monoid is $E$-unitary (see~\cite[Proposition 7.1.3]{Lawson}).

Given a set $X$, by $\cI X$ we denote the \emph{symmetric inverse monoid} on $X$. The elements of $\cI X$ are the partial bijections of $X$, and the product $fg$ of $f,g\in\cI X$ is the composition, i.e.
$$
fg:g\m(\dom f\cap\ran g)\to f(\dom f\cap\ran g),\ (fg)(x)=f(g(x)).
$$
Observe using~\cite[Proposition 1.1.4 (3)]{Lawson} that $f\le g$ in $\cI X$ if and only if $f\sst g$ regarded as subsets of $X\times X$. It follows that
\begin{align}\label{exists-Join-f_i-in-I(X)}
\mbox{there exists } \Jn_{i\in I}f_i\mbox{ in }\cI X\iff \bigcup_{i\in I} f_i\in \cI X,\mbox{ in which case }\Jn_{i\in I}f_i=\bigcup_{i\in I} f_i.
\end{align}
Here $\cup$ is the union of functions as subsets of $X\times X$.

Let $S$ and $T$ be inverse semigroups. Recall from~\cite{KL} that a map $\tau:S\to T$ is called a {\it premorphism}\footnote{In~\cite{Petrich} such a map is called a {\it prehomomorphism}, and in~\cite{LMS,St2001} it is called a {\it dual prehomomorphism}.}, if
\begin{enumerate}
	\item $\tau(s\m)=\tau(s)\m$;\label{tau(s-inv)=tau(s)-inv}
	\item $\tau(s)\tau(t)\le\tau(st)$.\label{tau(s)tau(t)<=tau(st)}
\end{enumerate}
One can easily show that 
\begin{align}\label{tau(E(S))-sst-E(T)}
\tau(E(S))\subseteq E(T)
\end{align}
(see, for example, \cite{St2001}). By a {\it partial action} of an inverse semigroup $S$ on a set $X$ we mean a premorphism $\tau:S\to\cI X$. If $\tau$ is a homomorphism, then we say that the partial action $\tau$ is {\it global}. We use the notation $\tau_s$ for the partial bijection $\tau(s)$. 

\section{Construction}

We fist make several observations. By \cite[Proposition 2.4.5]{Lawson} a congruence $\rho$ on $S$ is idempotent pure if and only if
\begin{align}\label{rho-sst-sim}
\rho\sst\sim,
\end{align}
where $\sim$ is the \emph{compatibility relation} on $S$ defined by 
\begin{align}\label{s-sim-t<=>s^(-1)t-and-st^(-1)-in-E(S)}
s\sim t\iff s\m t,st\m\in E(S).
\end{align}
Moreover, $\sim\sst\sigma$ by \cite[Theorem 2.4.1 (1)]{Lawson}, whence
\begin{align}\label{rho-sst-sigma}
\rho\sst\sigma
\end{align}
for each idempotent pure congruence $\rho$ on $S$.

We would also like to note that, given a partial action $\tau$ of $S$ on $X$ and $s\in S$, then by \ref{tau(s-inv)=tau(s)-inv} and \ref{tau(s)tau(t)<=tau(st)}
\begin{align}\label{id_ran-tau_s-sst-tau_ss^(-1)}
\id_{\ran{\tau_s}}=\tau_s\tau\m_s=\tau_s\tau_{s\m}\le\tau_{ss\m}.
\end{align}

\begin{lem}\label{corr-of-tau-tilde}
	Let $S$ be an inverse semigroup, $X$ a set, $\rho$ an idempotent pure congruence on $S$ and $\tau$ a partial action of $S$ on $X$. Then for any $\rho$-class $[s]\in S/\rho$ the join $\Jn_{t\in[s]}\tau_t$ exists in $\cI X$.
\end{lem}
\begin{proof}
	Let $(s,t)\in\rho$. Then by \ref{tau(s-inv)=tau(s)-inv} and \ref{tau(s)tau(t)<=tau(st)} one has that $\tau_s\tau\m_t=\tau_s\tau_{t\m}\le\tau_{st\m}$ and similarly $\tau\m_s\tau_t\le\tau_{s\m t}$. Since $\tau_{st\m},\tau_{s\m t}\in E(\cI X)$ thanks to~\eqref{tau(E(S))-sst-E(T)},~\eqref{rho-sst-sim} and~\eqref{s-sim-t<=>s^(-1)t-and-st^(-1)-in-E(S)}, it follows from~\cite[Proposition~1.2.1~(2)]{Lawson} that $\tau_s\cup\tau_t\in \cI X$. It remains to apply~\cite[Proposition~1.2.1~(3)]{Lawson} and~\eqref{exists-Join-f_i-in-I(X)}.
\end{proof}

\begin{lem}\label{tilde-tau-pact}
	Under the conditions of Lemma~\ref{corr-of-tau-tilde} define $\tl\tau:S/\rho\to\cI X$, $[s]\mapsto\tl\tau_{[s]}$, by
	\begin{align}\label{tl-tau-formula}
	\tl\tau_{[s]}=\Jn_{t\in[s]}\tau_t.
	\end{align}
	Then $\tl\tau$ is a partial action of $S/\rho$ on $X$.
\end{lem}
\begin{proof}
	By Lemma~\ref{corr-of-tau-tilde} the map $\tl\tau$ is well defined. Using \cite[Proposition 1.2.1 (4)]{Lawson} and \ref{tau(s-inv)=tau(s)-inv}, we obtain that
	\begin{align*}
	\tl\tau\m_{[s]}=\Jn_{t\in[s]}\tau\m_t=\Jn_{t\in[s]}\tau_{t\m}=\Jn_{t\in[s\m]}\tau_t=\tl\tau_{[s\m]}=\tl\tau_{[s]\m}.
	\end{align*}
	Moreover, applying \cite[Proposition 1.2.1 (5)]{Lawson} and \ref{tau(s)tau(t)<=tau(st)}, we conclude that
	\begin{align*}
	\tl\tau_{[s]}\tl\tau_{[t]}&=\left(\Jn_{s'\in[s]}\tau_{s'}\right)\left(\Jn_{t'\in[t]}\tau_{t'}\right)=\Jn_{s'\in[s]}\tau_{s'}\left(\Jn_{t'\in[t]}\tau_{t'}\right)=\Jn_{s'\in[s],\ t'\in[t]}\tau_{s'}\tau_{t'}\\
	&\le\Jn_{s'\in[s],\ t'\in[t]}\tau_{s't'}\le\Jn_{u\in[st]}\tau_u=\tl\tau_{[st]}=\tl\tau_{[s][t]}.
	\end{align*}
\end{proof}

\begin{rem}\label{S-E-unitary-rho=sigma}
	If, under the conditions of Lemma~\ref{tilde-tau-pact}, $S$ is $E$-unitary, $\rho=\sigma$ and $\bigcup_{e\in E(S)}\dom{\tau_e}=X$, then $\tl\tau$ is a partial action~\cite{E1,KL} of the maximum group image $\cG S=S/\sigma$ of $S$ on $X$  as in~\cite[Theorem 1.2]{St2001}. 
\end{rem}

\begin{rem}\label{S-F-inverse-rho=sigma}
	If, under the conditions of Lemma~\ref{tilde-tau-pact}, $S$ is an $F$-inverse monoid, $\tau$ is global and $\rho=\sigma$, then $\tl\tau_{[s]}=\tau_{\max[s]}$. In particular, if $G$ is a group, $S$ is the monoid $\cS G$ from~\cite{E1}\footnote{Notice that $\cS G$ coincides with Birget-Rhodes prefix expansion of $G$ (see~\cite{Birget-Rhodes89,Szendrei89})} and $\tau_{1_S}=\id _X$, then identifying $\cG S$ with $G$, we see that $\tl\tau$ is the partial action of $G$ on $X$ induced by $\tau$ as in~\cite[Theorem 4.2]{E1}.
\end{rem}

Let $(E,\le)$ be a meet semilattice, i.e. a partially ordered set, in which any pair of elements $e,f\in E$ has a meet $e\mt f$. Then $(E,\mt)$ is a commutative semigroup, and all its elements are idempotents. It follows that $E$ is inverse, and the natural partial order on $E$ coincides with $\le$. The converse is also true: each commutative inverse semigroup $S$, such that $E(S)=S$, is a meet semilattice, where $e\mt f=ef$ for all $e,f\in S$ (see~\cite[Proposition 1.4.9]{Lawson}). Notice that an order ideal of $(E,\le)$ is a semigroup ideal of $(E,\mt)$, an order isomorphism $(E_1,\le_1)\to(E_2,\le_2)$ is a semigroup isomorphism $(E_1,\mt_1)\to(E_2,\mt_2)$, and {\it vice versa}, so we may use the terms `ideal' and `isomorphism' in both senses.

Following~\cite{Petrich} for any meet semilattice $E$ we denote by $\Sg E$ the inverse subsemigroup of $\cI E$ consisting of the isomorphisms between ideals of $E$. Given an inverse semigroup $S$, by a {\it partial action $\tau$ of $S$ on $E$} we mean a premorphism $\tau:S\to\Sg E$. For any such $\tau$ set 
\begin{align}\label{E-rt-S-defn}
E\rt_\tau S=\{(e,s)\in E\times S\mid e\in\ran{\tau_s}\}
\end{align}
with the multiplication 
\begin{align}\label{(e_s)(e_t)-defn}
(e,s)(f,t)=(\tau_s(\tau\m_s(e)\mt f),st).
\end{align} 
Since $\tau_s(\tau\m_s(e)\mt f)\in\tau_s(\dom{\tau_s}\cap\ran{\tau_t})=\ran{(\tau_s\tau_t)}\subseteq\ran{\tau_{st}}$, the product~\eqref{(e_s)(e_t)-defn} is well defined.

\begin{lem}\label{L(E-tau-S)-inv-sem}
	The set $E\rt_\tau S$ is an inverse semigroup with respect to the operation~\eqref{(e_s)(e_t)-defn}. Its semilattice of idempotents is $\{(e,f)\in E\rt_\tau S\mid f\in E(S)\}$ and $(e,s)\m=(\tau\m_s(e),s\m)$ for every $(e,s)\in E\rt_\tau S$.
\end{lem}
\begin{proof}
	See~\cite[Lemmas VI.7.6--7]{Petrich}.
\end{proof}

\begin{lem}\label{L(E-tau-S)-mod-tilde-rho-isom}
	Let $\tau$ be a partial action of an inverse semigroup $S$ on a meet semilattice $E$ and $\rho$ an idempotent pure congruence on $S$. Then $\tl\tau$ is a partial action of $S/\rho$ on $E$.
\end{lem}
\begin{proof}
	We know from Lemma~\ref{tilde-tau-pact} that $\tl\tau$ is a premorphism from $S/\rho$ to $\cI E$. So, we only need to prove that $\tl\tau(S/\rho)\sst\Sg E$. Clearly, $\dom{\tl\tau_{[s]}}$ and $\ran{\tl\tau_{[s]}}$ are ideals of $E$ as unions of ideals $\dom{\tau_t}$ and $\ran{\tau_t}$, respectively, where $t\in[s]$. Furthermore, let $e,f\in\dom{\tl\tau_{[s]}}$ and $e\le f$. There exists $t\in[s]$, such that $f\in\dom{\tau_t}$. Since $\dom{\tau_t}$ is an ideal, it follows that $e\in\dom{\tau_t}$. Then, taking into account the fact that $\tau_t$ preserves $\le$, we have
	\begin{align*}
	\tl\tau_{[s]}(e)=\tau_t(e)\le\tau_t(f)=\tl\tau_{[s]}(f).
	\end{align*}
	Since $\tl\tau\m_{[s]}=\tl\tau_{[s]\m}$, then $\tl\tau\m_{[s]}$ is also order-preserving. Thus, $\tl\tau_{[s]}\in\Sg E$.
\end{proof}

\section{The embedding theorem}

Let $S$ be an inverse semigroup. It is well known (see, for example, \cite[Theorem~5.2.8]{Lawson}) that the map $\dt:S\to\Sg{E(S)}$, $s\mapsto\dt_s$, where
\begin{align}\label{dl_s-defn}
\dt_s:s\m s E(S)\to ss\m E(S)\mbox{ and }\dt_s(e)=ses\m\mbox{ for any }e\in s\m sE(S),
\end{align}
is a global action of $S$ on $E(S)$, called the \emph{Munn representation} of $S$.

\begin{prop}\label{S-embeds-into-E(S)*S}
	Let $S$ be an inverse semigroup and $\rho$ an idempotent pure congruence on $S$. Then $S$ embeds into $E(S)\rt_{\tl\dt}(S/\rho)$.
\end{prop}
\begin{proof}
	Define $\varphi:S\to E(S)\rt_{\tl\dt}(S/\rho)$ by 
	\begin{align}\label{phi(s)=(ss^(-1)_s)}
	\varphi(s)=(ss\m,[s]).
	\end{align}
	As $ss\m\in ss\m E(S)\subseteq\ran{\tl\dt_{[s]}}$, the map $\varphi$ is well defined. Since by~\eqref{tl-tau-formula} and~\eqref{dl_s-defn}
	\begin{align*}
	(ss\m,[s])(tt\m,[t])&=(\tl\dt_{[s]}(\tl\dt_{[s]\m}(ss\m)tt\m),[st])\\
	&=(\tl\dt_{[s]}(\dt_{s\m}(ss\m)tt\m),[st])\\
	&=(\tl\dt_{[s]}(s\m s\cdot tt\m),[st])\\
	&=(\dt_s(s\m s\cdot tt\m),[st])\\
	&=(stt\m s\m,[st])\\
	&=(st(st)\m,[s][t]),
	\end{align*}
	$\varphi$ is a homomorphism. Suppose that $\varphi(s)=\varphi(t)$. By~\eqref{phi(s)=(ss^(-1)_s)} this means that $(s,t)\in\cR\cap\rho$, where $\cR$ is the Green's relation on $S$, defined by 
	\begin{align*}
	(s,t)\in\cR\iff ss\m=tt\m.
	\end{align*}
	As $\cR\cap\rho\sst\cR\cap\sim$, it follows by~\cite[Lemma III.2.13]{Petrich} that $s=t$,
	and hence $\varphi$ is one-to-one.
\end{proof}

To make the result of Proposition~\ref{S-embeds-into-E(S)*S} more precise, we follow \cite{O'Carroll77}. We say that a partial action $\tau$ of $S$ on $E$ is {\it strict}, if for each $e\in E$ the set $\{f\in E(S)\mid e\in\dom{\tau_f}\}$ has a minimum element, denoted by $\alpha(e)$, and the map $\alpha:E\to E(S)$ is a homomorphism of meet semilattices. Observe using~\eqref{id_ran-tau_s-sst-tau_ss^(-1)} that if $(e,s)\in E\rt_\tau S$, then $\alpha(e)\le ss\m$. The set of $(e,s)\in E\rt_\tau S$ for which 
\begin{align}\label{af(e)=ss^(-1)}
\alpha(e)=ss\m
\end{align}
will be denoted by $E\rt^m_\tau S$.

\begin{lem}\label{L_m(E-tau-S)-inv-sem}
	The subset $E\rt^m_\tau S$ is an inverse subsemigroup of $E\rt_\tau S$.
\end{lem}
\begin{proof}
	We first observe, as in \cite[Lemma 2]{O'Carroll77}, that 
	\begin{align}\label{alpha(tau_s(e))}
	\alpha(\tau_s(e))=s\alpha(e)s\m  
	\end{align}
	for any $e\in\dom{\tau_s}$. Indeed, $e\in\dom{\tau_{\alpha(e)}}$, therefore using~\ref{tau(s)tau(t)<=tau(st)} and~\eqref{id_ran-tau_s-sst-tau_ss^(-1)} we have
	\begin{align*}
	\tau_s(e)\in\ran{(\tau_s\tau_{\alpha(e)})}\subseteq\ran{\tau_{s\alpha(e)}}\subseteq\dom{\tau_{s\alpha(e)(s\alpha(e))\m}}=\dom{\tau_{s\alpha(e)s\m}},
	\end{align*}
	whence
	\begin{align}\label{af(tau_s(e))<=s-af(e)s^(-1)}
	\alpha(\tau_s(e))\le s\alpha(e)s\m.
	\end{align}
	For the converse inequality replace $e$ by $\tau_s(e)$ and $s$ by $s\m$ in~\eqref{af(tau_s(e))<=s-af(e)s^(-1)} and then use \ref{tau(s-inv)=tau(s)-inv}.
	
	Let $(e,s),(f,t)\in E\rt^m_\tau S$. Then for their product $(\tau_s(\tau\m_s(e)\mt f),st)$ we have by \ref{tau(s-inv)=tau(s)-inv},~\eqref{af(e)=ss^(-1)} and~\eqref{alpha(tau_s(e))} that
	$$
	\alpha(\tau_s(\tau\m_s(e)\mt f))=s\alpha(\tau\m_s(e)\mt f)s\m=s\cdot s\m\alpha(e)s\cdot\alpha(f)s\m=stt\m s\m.
	$$
	Hence, $(e,s)(f,t)\in E\rt^m_\tau S$. Moreover, for each $(e,s)\in E\rt^m_\tau S$ the inverse $(e,s)\m=(\tau\m_s(e),s\m)$ belongs to $E\rt^m_\tau S$, as $\alpha(\tau\m_s(e))=s\m\alpha(e)s=s\m s$.
\end{proof}

A strict partial action $\tau$ of $S$ on $E$ is called {\it fully strict}, if the homomorphism 
\begin{align}\label{pi(e_s)-mapsto-s}
\pi:E\rt^m_\tau S\to S,\ (e,s)\mapsto s,
\end{align}
is surjective.

\begin{rem}\label{L_m=L-iff-S-group}
	Let $\tau$ be a fully strict partial action of $S$ on $E$. Then $E\rt^m_\tau S=E\rt_\tau S$ if and only if $S$ is a group.
\end{rem}
\noindent Indeed, the ``if'' part trivially holds for any strict $\tau$. For the ``only if'' part take $s\in S$ and suppose that $s\le t$ for some $t\in S$. Since $\tau$ is fully strict, there are $e,f\in E$, such that $(e,s),(f,t)\in E\rt^m_\tau S$. In particular, $e\in\ran{\tau_s}$ and $f\in\ran{\tau_t}$. Since $\ran{\tau_s}$ and $\ran{\tau_t}$ are ideals of $E$, we have $ef\in\ran{\tau_s}\cap\ran{\tau_t}$, whence $(ef,s),(ef,t)\in E\rt_\tau S$. But $E\rt_\tau S=E\rt^m_\tau S$, so $ss\m=\alpha(ef)=tt\m$ by~\eqref{af(e)=ss^(-1)}. Consequently, $(s,t)\in\cR$. Then $s=t$ thanks to~\cite[Proposition~3.2.3 (2)]{Lawson}, and thus the natural partial order on $S$ is trivial. It remains to apply~\cite[Proposition~1.4.10]{Lawson}.

\begin{thrm}\label{S-embeds-into-L-precise}
	Given an inverse semigroup $S$ and an idempotent pure congruence $\rho$ on $S$, there exists a fully strict partial action $\tau$ of $S/\rho$ on $E(S)$, such that $S$ embeds into $E(S)\rt_\tau(S/\rho)$. Moreover, the image of $S$ in $E(S)\rt_\tau(S/\rho)$ coincides with $E(S)\rt^m_\tau(S/\rho)$, and $\rho$ is induced by the epimorphism $\pi:E(S)\rt^m_\tau(S/\rho)\to S/\rho$. In particular, the embedding is surjective if and only if $\rho$ is a group congruence, i.e. $S$ is $E$-unitary and $\rho=\sigma$.
\end{thrm}
\begin{proof}
	As we know from Proposition~\ref{S-embeds-into-E(S)*S}, there is an embedding $\varphi:S\to E(S)\rt_{\tl\dt}(S/\rho)$ given by~\eqref{phi(s)=(ss^(-1)_s)}. For each $e\in E(S)$ consider the set 
	\begin{align*}
	A_e=\left\{[f]\in E(S/\rho)\mid e\in\dom{\tau_{[f]}}\right\}.
	\end{align*}
	Clearly, $[e]\in A_e$, as $e\in eE(S)=\dom{\dt_e}\sst\dom{\tl\dt_{[e]}}$. Moreover, if $[f]\in A_e$, then $e\in gE(S)$ for some idempotent $g\in[f]$. So, $e\le g$ and hence $[e]\le[g]=[f]$. This shows that $[e]$ is the minimum element of $A_e$, and since $\alpha(e)=[e]$ is a homomorphism $E(S)\to E(S/\rho)$, the partial action $\tl\dt$ is strict. Now, for each $s\in S$ consider the pair $\varphi(s)=(ss\m,[s])\in E(S)\rt_{\tl\dt}(S/\rho)$. We have $\alpha(ss\m)=[ss\m]=[s][s]\m$, so in fact $(ss\m,[s])\in E(S)\rt^m_{\tl\dt}(S/\rho)$, proving that $\tl\dt$ is fully strict.
	
	The range of the embedding $\varphi$ is contained in $E(S)\rt^m_{\tl\dt}(S/\rho)$ as shown above. Now if $(e,[s])\in E(S)\rt^m_{\tl\dt}(S/\rho)$, then $e\in tt\m E(S)$ for some $t\in[s]$ and $[e]=\alpha(e)=[s][s]\m$. It follows that $e=ett\m=et(et)\m$ and $[et]=[s][s]\m[t]=[s][s]\m[s]=[s]$. Therefore, $(e,[s])=\varphi(et)$, whence $\varphi(S)=E(S)\rt_{\tl\dt}^m(S/\rho)$. Moreover, $\pi(\varphi(s))=\pi(\varphi(t))\iff[s]=[t]\iff (s,t)\in\rho$ in view of~\eqref{phi(s)=(ss^(-1)_s)} and~\eqref{pi(e_s)-mapsto-s}.
	
	Finally, if $\varphi$ is surjective, then $E(S)\rt^m_{\tl\dt}(S/\rho)=E(S)\rt_{\tl\dt}(S/\rho)$, so $S/\rho$ is a group by Remark~\ref{L_m=L-iff-S-group}. Therefore, $\rho\supseteq\sigma$ by \cite[Theorem 2.4.1 (3)]{Lawson}, which yields $\rho=\sigma$ in view of \ref{rho-sst-sigma}, and thus $S$ is $E$-unitary. 
\end{proof}

\section{Globalizable partial actions and O'Carroll $L$-triples}\label{glob-sec}
Let $S$ be an inverse semigroup, $\phi:S\to\cI X$, $s\mapsto\phi_s$, a partial action of $S$ on a set $X$ and $Y\sst X$. Then for any $s\in S$ the partial bijection
\begin{align}\label{tau_s=id_Y-phi_s-id_Y}
\tau_s=\id_Y\phi_s\id_Y
\end{align}
is the restriction of $\phi_s$ to the subset 
\begin{align}\label{dom-restr-phi_s}
\dom{\tau_s}=\{y\in \dom{\phi_s}\cap Y\mid \phi_s(y)\in Y\}=\phi\m_s(\ran{\phi_s}\cap Y)\cap Y\sst\dom{\phi_s}.
\end{align}
Thus, $\tau_s$ is a bijection 
\begin{align*}
\phi\m_s(\ran{\phi_s}\cap Y)\cap Y\to \phi_s(\dom{\phi_s}\cap Y)\cap Y.
\end{align*}
Moreover, 
\begin{align*}
\tau_s\tau_t=\id_Y\phi_s\id_Y\phi_t\id_Y\le\id_Y\phi_s\phi_t\id_Y\le\id_Y\phi_{st}\id_Y=\tau_{st},
\end{align*}
so that $\tau:S\to\cI Y$, $s\mapsto\tau_s$, is a partial action of $S$ on $Y$ called the \emph{restriction} of $\phi$ to $Y$.

A partial action $\tau$ of an inverse semigroup $S$ on a set $Y$ is said to be \emph{globalizable} if it admits a \emph{globalization}\footnote{In a more general context this was called an \emph{augmented action} in \cite{Gould-Hollings09}}, i.e. a global action $\phi$ of $S$ on a set $X$ together with an injective map $\io:Y\to X$, such that
\begin{align}\label{tau_s=i^(-1)-phi_s-i}
\tau_s=\bar\io\phi_s\bar{\io}\m
\end{align}
holds for all $s\in S$. Here $\bar{\io}$ is the bijection $Y\to\io(Y)$, $y\mapsto\io(y)$. It was proved in \cite[Theorem 6.10]{Gould-Hollings09} that $\tau$ is globalizable provided that it is order-preserving, i.e. $s\le t\impl\tau_s\le\tau_t$ for all $s,t\in S$. Observe that the converse of this fact is also true. For, a global action $\phi$ of $S$ on $X$, being a homomorphism $S\to\cI X$, is order-preserving, so that its restriction $\tau$ defined by \eqref{tau_s=id_Y-phi_s-id_Y} is also order-preserving.

Globalizable partial actions of inverse semigroups on semilattices turn out to be closely related to O'Carroll $L$-triples. Recall from~\cite{O'Carroll77} that an \emph{$L$-triple} is $(T,X,Y)$, where 
\begin{enumerate}
	\item $T$ is an inverse semigroup, $X$ is a down-directed poset, $Y$ is a meet subsemilattice and order ideal of $X$;
	\item $T$ acts (globally) on $X$ via a homomorphism $\phi:T\to\cI X$, $t\mapsto\phi_t$, such that $\phi_t$ is an order isomorphism between	non-empty order ideals of $X$;
	\item $X=TY$, where $TY=\bigcup_{t\in T}\phi_t(\dom{\phi_t}\cap Y)$.
\end{enumerate}
With any $L$-triple $(T,X,Y)$ O'Carroll associates in~\cite{O'Carroll77} the inverse semigroup
\begin{align}\label{L(T_X_Y)-defn}
L(T,X,Y)=\{(a,t)\in X\times T\mid a\in Y\cap\ran{\phi_t},\ \phi\m_t(a)\in Y\},
\end{align}
where the product of any two pairs from $L(T,X,Y)$ is given by the same formula~\eqref{(e_s)(e_t)-defn} as we used in the case of a partial action. An $L$-triple $(T,X,Y)$ was called \emph{strict} in~\cite{O'Carroll77} if for every $y\in Y$ the set $\{f\in E(T)\mid y\in\dom{\phi_f}\}$ has a minimum element $e(y)$, and the map $e:Y\to E(T)$, $y\mapsto e(y)$, is a homomorphism of meet semilattices. In this case 
\begin{align*}
L_m(T,X,Y)=\{(a,t)\in L(T,X,Y)\mid e(a)=tt\m\}
\end{align*}
is an inverse subsemigroup of $L(T,X,Y)$ by~\cite[Theorem~2]{O'Carroll77}. Moreover, if the homomorphism $L_m(T,X,Y)\to T$, $(a,t)\mapsto t$, is surjective, then the strict $L$-triple $(T,X,Y)$ was said to be \emph{fully strict}~\cite{O'Carroll77}.

\begin{prop}\label{from-L-triple-to-part-act}
	Let $(T,X,Y)$ be an $L$-triple and $\phi:T\to\cI X$ the corresponding action of $T$ on $X$. Denote by $\tau$ the restriction of $\phi$ to $Y$. Then $\tau$ is a partial action of $T$ on the meet semilattice $Y$ such that $\dom{\tau_t}\ne\emptyset$ for all $t\in T$, and $L(T,X,Y)=Y\rt_\tau T$. Moreover, $\tau$ is strict (resp. fully strict) if and only if $(T,X,Y)$ is strict (resp. fully strict), in which case $L_m(T,X,Y)=Y\rt_\tau^m T$.
\end{prop}
\begin{proof}
	It was shown in~\cite[Lemma~3]{O'Carroll76} that, given an action $\phi$ of $T$ on $(X,\le)$, such that $\phi_t$ is an order isomorphism between order ideals of $X$ for all $t\in T$, and an order ideal $Y$ of $X$ satisfying $TY=X$, one has that
	\begin{align}\label{X-down-dir<=>dom-tau_t-non-empty}
	X\mbox{ is down-directed }\&\ \forall t\in T:\dom\phi_t\ne\emptyset\iff\forall t\in T:\phi_t(\dom\phi_t\cap Y)\cap Y\ne\emptyset.
	\end{align}
	By~\eqref{dom-restr-phi_s} the latter is the same as $\forall t\in T:\ran{\tau_t}\ne\emptyset$ or, equivalently, $\forall t\in T:\dom{\tau_t}\ne\emptyset$. Let $y\in\dom{\tau_t}$, i.e. $y\in\dom{\phi_t}\cap Y$ and $\phi_t(y)\in Y$. For any $z\le y$ one has that $z\in\dom{\phi_t}\cap Y$, as $Y$ and $\dom{\phi_t}$ are ideals of $X$. Moreover, $\phi_t(z)\le \phi_t(y)\in Y$  since $\phi_t$ is order-preserving, so $\phi_t(z)\in Y$. Thus, $z\in\dom{\tau_t}$, and consequently $\tau_t\in\Sg{Y}$ for all $t\in T$.
	
	Notice using~\eqref{dom-restr-phi_s} and~\eqref{L(T_X_Y)-defn} that $(a,t)\in L(T,X,Y)\iff a\in \ran{\tau_t}$. Thus, $L(T,X,Y)$ is exactly $Y\rt_\tau T$ defined by~\eqref{E-rt-S-defn}. Now, for any $y\in Y$ and $f\in E(T)$
	\begin{align*}
	y\in\dom{\phi_f}\iff y\in\dom{\phi_f}\cap Y=\phi_f(\dom{\phi_f}\cap Y)\cap Y=\dom{\tau_f}.	
	\end{align*}
	Hence, $(T,X,Y)$ is strict if and only if $\tau$ is strict. The remaining assertions of the proposition are immediate. 
\end{proof}

\begin{lem}\label{glob-with-T-iota(Y)=X}
	Let $\tau$ be a strict partial action of an inverse semigroup $T$ on a semilattice $(Y,\le)$. If $\tau$ is globalizable, then $\tau$ admits a globalization $\phi:T\to \cI X$ such that $T\io(Y)=X$, where $\io:Y\to X$ is the corresponding injective map.
\end{lem}
\begin{proof}
	Let $\phi':T\to\cI {X'}$, $\io:Y\to X'$ be a globalization of $\tau$. Without loss of generality we may identify $Y$ with $\io(Y)$, so that $Y\sst X'$, $\io:Y\to X'$ is the inclusion map and $\tau$ is the restriction of $\phi'$ to $Y$. Consider 
	\begin{align}\label{X-is-TY}
	X=TY\sst X'\mbox{ and }\phi\mbox{ being the restriction of }\phi'\mbox{ to }X.
	\end{align}
	Since $\tau$ is strict, each $y\in Y$ belongs to 
	\begin{align*}
	\dom{\tau_{\af(y)}}=\ran{\tau_{\af(y)}}=\phi'_{\af(y)}(\dom{\phi'_{\af(y)}}\cap Y)\cap Y\sst X,
	\end{align*}
	so $Y\sst X$. Moreover, notice by~\eqref{X-is-TY} that
	\begin{align}\label{phi_t(dom-phi'_t-cap-TY)-sst-TY}
	x\in\dom{\phi'_t}\cap X\impl \phi'_t(x)\in X.
	\end{align}
	Indeed, if $x\in\dom{\phi'_t}$ and $x=\phi'_u(y)$ for some $u\in T$ and $y\in\dom{\phi'_u}\cap Y$, then $\phi'_t(x)=\phi'_{tu}(y)$. This implies that
	\begin{align}\label{dom-phi_t=dom-phi'_t-cap-X}
	\dom{\phi_t}=\dom{\phi'_t}\cap X.
	\end{align}
	It follows that~\eqref{tau_s=i^(-1)-phi_s-i} holds, and we only need to show that the partial action $\phi$ is global, i.e. $\phi_{tu}=\phi_t\phi_u$ for all $t,u\in T$. Using the fact that $\phi'$ is global,~\eqref{phi_t(dom-phi'_t-cap-TY)-sst-TY} and~\eqref{dom-phi_t=dom-phi'_t-cap-X}, we have
	\begin{align*}
	y&\in\dom{\phi_{tu}}=\dom{\phi'_{tu}}\cap X=\dom{\phi'_t\phi'_u}\cap X\\
	&\iff y\in\dom{\phi'_u}\cap X\ \&\ \phi'_u(y)\in\dom{\phi'_t}\\
	&\iff y\in\dom{\phi'_u}\cap X\ \&\ \phi'_u(y)\in\dom{\phi'_t}\cap X\\
	&\iff y\in\dom{\phi_u}\ \&\ \phi_u(y)\in\dom{\phi_t}\\
	&\iff y\in\dom{\phi_t\phi_u},
	\end{align*}
	in which case $\phi_{tu}(y)=\phi'_{tu}(y)=\phi'_t(\phi'_u(y))=\phi_t(\phi_u(y))$.
\end{proof}

\begin{lem}\label{part-order-on-TY}
	Let $\tau$ be a strict partial action of an inverse semigroup $T$ on a semilattice $(Y,\le)$ and $\phi:T\to \cI X$ a globalization of $\tau$ such that $T\io(Y)=X$, where $\io:Y\to X$ is the corresponding injective map. Then there exists a partial order $\le'$ on $X$, such that $\phi_t:\dom{\phi_t}\to\ran{\phi_t}$ is an order isomorphism between order ideals of $X$. Moreover, $(\io(Y),\le')$ is an order ideal and meet subsemilattice of $(X,\le')$ isomorphic to $(Y,\le)$.
\end{lem}
\begin{proof}
	As in the proof of Lemma~\ref{glob-with-T-iota(Y)=X} we assume for simplicity that $Y\sst X$ and $\io:Y\to X$ is the inclusion map. Define the relation $\le'$ on $X$ as follows
	\begin{align}\label{x_1-le-x_2-in-X}
	x_1\le' x_2&\iff\exists t\in T\mbox{ and }\exists y_1,y_2\in\dom{\phi_t}\cap Y\mbox{ such that }\notag\\
	&x_1=\phi_t(y_1),\ x_2=\phi_t(y_2)\mbox{ and }y_1\le y_2.
	\end{align}
	Reflexivity of $\le'$ is explained by the fact that $TY=X$. For anti-symmetry suppose that $x_1\le' x_2$ and $x_2\le' x_1$, i.e. $x_1=\phi_t(y_1)=\phi_u(z_1)$ and $x_2=\phi_t(y_2)=\phi_u(z_2)$ for some $t,u\in T$ and $y_1,y_2\in\dom{\phi_t}\cap Y$, $z_1,z_2\in \dom{\phi_u}\cap Y$, such that $y_1\le y_2$ and $z_2\le z_1$. Then $z_1=\phi\m_u(\phi_t(y_1))=\phi_{u\m t}(y_1)=\tau_{u\m t}(y_1)$ and similarly $z_2=\tau_{u\m t}(y_2)$. Since $y_1\le y_2$ and $\tau_{u\m t}$ is order-preserving, we conclude that $z_1\le z_2$, whence $z_1=z_2$ by anti-symmetry of $\le$, and thus $x_1=x_2$. For transitivity of $\le'$ take $x_1\le' x_2$ and $x_2\le' x_3$. By~\eqref{x_1-le-x_2-in-X} there are $t,u\in T$ and $y_1,y_2\in\dom{\phi_t}\cap Y$, $z_1,z_2\in \dom{\phi_u}\cap Y$, such that 
	\begin{align}
	x_1&=\phi_t(y_1),\label{x_1-is-phi_t(y_1)}\\
	x_2&=\phi_t(y_2)=\phi_u(z_1),\label{x_2-is-phi_t(y_2)-is-phi_u(z_1)}\\
	x_3&=\phi_u(z_2)\label{x_3-is-phi_u(z_2)}
	\end{align}
	and $y_1\le y_2$, $z_1\le z_2$. It follows from~\eqref{x_2-is-phi_t(y_2)-is-phi_u(z_1)} that $y_2=\phi_{t\m u}(z_1)=\tau_{t\m u}(z_1)$. Since $\ran{\tau_{t\m u}}$ is an ideal of $Y$ and $y_1\le y_2$, one has that $y_1=\tau_{t\m u}(z_0)$ for some $z_0\in\dom{\tau_{t\m u}}$. Applying $\tau\m_{t\m u}$ to the inequality $y_1\le y_2$, we obtain $z_0\le z_1$, whence $z_0\le z_2$ by transitivity of $\le$. Using~\eqref{x_1-is-phi_t(y_1)}, we have $\phi_u(z_0)=\phi_u(\tau\m_{t\m u}(y_1))=\phi_u(\phi\m_{t\m u}(y_1))=\phi_{uu\m}(\phi_t(y_1))=\phi_{uu\m}(x_1)=x_1$. Then $x_1\le' x_3$ in view of~\eqref{x_3-is-phi_u(z_2)}, which completes the proof of transitivity of $\le'$, and thus $\le'$ is a partial order.
	
	Suppose that $x_2\in\dom{\phi_t}$ and $x_1\le'x_2$. By~\eqref{x_1-le-x_2-in-X} there are $u\in T$ and $y_1,y_2\in\dom{\phi_u}\cap Y$, such that $x_1=\phi_u(y_1)$, $x_2=\phi_u(y_2)$ and $y_1\le y_2$. It follows that $y_2\in\dom{\phi_{tu}}=\dom{\phi_{(tu)\m tu}}$. Since $\phi_{(tu)\m tu}(y_2)=y_2\in Y$, we have that $y_2\in\dom{\tau_{(tu)\m tu}}$. Then $\af(y_2)\le (tu)\m tu$. But $\af$, being a homomorphism of meet semilattices $Y\to E(T)$, is order-preserving, so $\af(y_1)\le\af(y_2)$ in $E(T)$, and by transitivity $\af(y_1)\le(tu)\m tu$. Therefore, $\phi_{\af(y_1)}\le\phi_{(tu)\m tu}=\phi\m_{tu}\phi_{tu}$. As $y_1\in\dom{\tau_{\af(y_1)}}\sst\dom{\phi_{\af(y_1)}}$, we conclude that $y_1\in\dom{(\phi\m_{tu}\phi_{tu})}=\dom{\phi_{tu}}=\dom{(\phi_t\phi_u)}$. The latter implies that $x_1=\phi_u(y)\in\dom{\phi_t}$, so that $\dom{\phi_t}$ is an order ideal of $(X,\le')$. Notice also that in this case $\phi_t(x_1)=\phi_{tu}(y_1)$ and $\phi_t(x_2)=\phi_{tu}(y_2)$. Hence $\phi_t(x_1)\le'\phi_t(x_2)$, and $\phi_t$ is an order homomorphism $\dom{\phi_t}\to\ran{\phi_t}$. Since the same holds for $\phi_{t\m}=\phi\m_t$, the bijection $\phi_t$ is an order isomorphism between order ideals of $(X,\le')$.
	
	Let $y_1,y_2\in Y$ with $y_1\le y_2$. Since $y_1=\tau_{\af(y_1)}(y_1)$, $y_2=\tau_{\af(y_2)}(y_2)$ and $\ran{\tau_{\af(y_1)}}$, $\ran{\tau_{\af(y_2)}}$ are ideals of $Y$, we have that $y_1,y_2\in\ran{\tau_{\af(y_1)}}\cap\ran{\tau_{\af(y_2)}}\sst\ran{\phi_{\af(y_1)}}\cap\ran{\phi_{\af(y_2)}}=\ran{\phi_{\af(y_1\mt y_2)}}$, whence $y_1=\phi_{\af(y_1\mt y_2)}(y_1)$ and $y_2=\phi_{\af(y_1\mt y_2)}(y_2)$. Therefore, $y_1\le' y_2$. Conversely, suppose that $y_1,y_2\in Y$ with $y_1\le' y_2$. By~\eqref{x_1-le-x_2-in-X} there are $t\in T$ and $z_1,z_2\in Y\cap\dom{\phi_t}$, such that $y_1=\phi_t(z_1)$, $y_2=\phi_t(z_2)$ and $z_1\le z_2$. Observe that $y_1=\tau_t(z_1)$, $y_2=\tau_t(z_2)$, and consequently $y_1\le y_2$, as $\tau_t$ is order-preserving. This shows that the intersection of $\le'$ with $Y\times Y$ coincides with $\le$. In particular, $(Y,\le')=(Y,\le)$ is a meet subsemilattice of $(X,\le')$.
	
	We now prove that $Y$ is an ideal of $X$. Suppose that $y\in Y$ and $x\le' y$ for some $x\in X$. There are $t\in T$ and $y_1,y_2\in Y\cap\dom{\phi_t}$ with $y_1\le y_2$, such that $x=\phi_t(y_1)$ and $y=\phi_t(y_2)$. Then $y_2=\tau\m_t(y)$, and since $\dom{\tau_t}$ is an ideal of $Y$, $y_1\in \dom{\tau_t}$, so that $x=\tau_t(y_1)\in Y$.
\end{proof}

\begin{prop}\label{from-part-act-to-L-triple}
	Let $\tau$ be a globalizable strict partial action of an inverse semigroup $T$ on a semilattice $Y$ such that $\dom{\tau_t}\ne\emptyset$ for all $t\in T$. Then there exists a globalization $\phi:T\to \cI X$, $\io:Y\to X$ of $\tau$ such that $(T,X,\io(Y))$ is a strict $L$-triple. Moreover, in this case $Y\rt_\tau T\cong L(T,X,\io(Y))$ and $Y\rt_\tau^m T\cong L_m(T,X,\io(Y))$. In particular, if $\tau$ is fully strict, then $\dom{\tau_t}\ne\emptyset$ is automatically satisfied for all $t\in T$, and in this case $(T,X,\io(Y))$ is also fully strict.
\end{prop}
\begin{proof}
	Existence of a pair $\phi:T\to \cI X$, $\io:Y\to X$, such that $(T,X,\io(Y))$ is an $L$-triple, follows from Lemmas~\ref{glob-with-T-iota(Y)=X} and~\ref{part-order-on-TY} and the observation at the beginning of the proof of Proposition~\ref{from-L-triple-to-part-act} (see~\eqref{X-down-dir<=>dom-tau_t-non-empty}). The strictness of $(T,X,\io(Y))$ is guaranteed by the strictness of $\tau$: the corresponding homomorphism $e:\io(Y)\to E(T)$ is just the (usual) composition of $\af$ with $\bar\io\m$. The isomorphism $Y\rt_\tau T\cong L(T,X,\io(Y))$ is simply the identification of $(y,t)\in Y\rt_\tau T$ with $(\io(y),t)\in L(T,X,\io(Y))$, which restricts to the isomorphism $Y\rt_\tau^m T\cong L_m(T,X,\io(Y))$. Finally, if $\tau$ is fully strict, then for every $t\in T$ there exists $y\in Y$, such that $(y,t)\in Y\rt_\tau^m T$. In particular, $y\in\ran{\tau_t}$, so that $\dom{\tau_t}=\ran{\tau_{t\m}}\ne\emptyset$ for all $t\in T$. The full strictness of $(T,X,\io(Y))$, whenever $\tau$ is fully strict, is immediate.
\end{proof}

Propositions~\ref{from-L-triple-to-part-act} and~\ref{from-part-act-to-L-triple} permit us to conclude that O'Carroll's~\cite[Theorem 4]{O'Carroll77} is equivalent to the version of our Theorem~\ref{S-embeds-into-L-precise} where the partial action $\tau$ is additionally globalizable. The latter property, however, does not always hold for a partial action, as the following example shows.
\begin{exm}\label{non-global-tl-dt}
	Let $S$ be the meet semilattice $\{0,e,f\}$, where $0$ is the minimum element and $e\mt f=0$. Consider the least congruence $\rho$ on $S$ which contains the pair $(0,e)$. Then $\rho$ is idempotent pure and the partial action $\tl\dt$ of $S/\rho$ on $E(S)$  is not globalizable.
\end{exm}
\begin{proof}
	Indeed, the fact that $\rho$ is idempotent pure is trivial, since $S=E(S)$. Notice that the $\rho$-classes are $[e]=\{e,0\}$ and $[f]=\{f\}$, so that $[e]\le[f]$. Now,
	\begin{align*}
	\tl\dt_{[e]}&=\dt_e\jn\dt_0=\id_{eE(S)}\jn\id_{0E(S)}=\id_{\{0,e\}}\jn\id_{\{0\}}=\id_{\{0,e\}},\\
	\tl\dt_{[f]}&=\dt_f=\id_{fE(S)}=\id_{\{0,f\}}.
	\end{align*}
	Since $\tl\dt_{[e]}\not\le\tl\dt_{[f]}$, the premorphism $\tl\dt$ is not order-preserving. Thus, it is not globalizable as a partial action.
\end{proof}

\section*{Acknowledgements}
	The author thanks the referee for the very detailed reading of the manuscript and numerous useful suggestions which permitted to simplify and shorten the proofs. In particular, the use of~\cite[Proposition 1.2.1]{Lawson} in the proofs of Lemmas~\ref{corr-of-tau-tilde} and~\ref{tilde-tau-pact}, as well as the use of the order-preserving property of $\tau_t$ in the proof of Lemma~\ref{L(E-tau-S)-mod-tilde-rho-isom} are due to the referee. Section~\ref{glob-sec} also arose from referee's comments.

	\bibliography{bibl-pact}{}
	\bibliographystyle{acm}
\end{document}